\documentclass[a4paper]{baustmsmodified}

\pdfoutput=1 

\usepackage{amssymb,amsthm}
\usepackage{amsfonts}
\usepackage{amsmath}
\usepackage[british]{babel}
\usepackage{cleveref}
\usepackage[shortlabels]{enumitem}
\usepackage{mathtools}
\usepackage{cancel}
\usepackage{stmaryrd}
\SetSymbolFont{stmry}{bold}{U}{stmry}{m}{n}
\usepackage{tikz}
\usetikzlibrary{calc}
\usepackage{float}
\usepackage{latexsym}
\usepackage{etoolbox}
\usepackage{microtype}

\DeclareMathOperator{\A}{A}
\DeclareMathOperator{\D}{D}

\DeclareMathOperator{\R}{R}

\newcommand\down{\mathord\downarrow}

\DeclareMathOperator{\minusset}{/}

\newcommand \compo {\mathbin{;}}

\newcommand \bmeet {\cdot}
\newcommand \bjoin {+}

\newcommand \limplies {\mathbin{\rightarrow}}

\newcommand \aand \wedge

\newcommand{\join}{\sum}

\DeclareMathOperator{\kstar}{*}

\DeclareMathOperator{\inv}{^{-1}}

\newcommand{\from}{\colon}

\newcommand{\interp}[1]{\llbracket#1\rrbracket}
\newcommand{\tinterp}[1]{[#1]}
\newcommand{\algebra}[1]{\mathfrak{#1}}

\newcommand{\defn}[1]{\textbf{#1}}

\newcommand\trivial{\vec \varepsilon}

\makeatletter
\newcommand\mynobreakpar{\par\nobreak\@afterheading}
\makeatother

\renewcommand{\vec}[1]{\boldsymbol{#1}}
\renewcommand{\>}{\rangle}

\theoremstyle{plain}
\newtheorem{theorem}{Theorem}[section]
\newtheorem{proposition}[theorem]{Proposition}
\newtheorem{corollary}[theorem]{Corollary}
\newtheorem{lemma}[theorem]{Lemma}

\theoremstyle{definition}
\newtheorem{definition}[theorem]{Definition}
\newtheorem{terminology}[theorem]{Terminology}
\newtheorem{notation}[theorem]{Notation}

\theoremstyle{remark}
\newtheorem{remark}[theorem]{Remark}
\newtheorem{problem}[theorem]{Problem}
\newtheorem{example}[theorem]{Example}

\numberwithin{equation}{section}

\begin{document}
\runningtitle{Free Kleene algebras with domain}
\title{Free Kleene algebras with domain}
\author[1]{Brett McLean}
\address[1]{Laboratoire J. A. Dieudonn\'e UMR CNRS 7351, Universit\'e Nice Sophia Antipolis, 06108 Nice Cedex 02\email{brett.mclean@unice.fr}}

\authorheadline{Brett McLean}


\support{This project has received funding from the European Research Council (ERC) under the European Union's
Horizon 2020 research and innovation program (grant agreement No. 670624).}



\begin{abstract}
First we identify the free algebras of the class of algebras of binary relations equipped with the composition and domain operations. Elements of the free algebras are pointed labelled finite rooted trees. Then we extend to the analogous case when the signature includes all the Kleene algebra with domain operations; that is, we add union and reflexive transitive closure to the signature. In this second case, elements of the free algebras are `regular' sets of the trees of the first case. As a corollary, the axioms of domain semirings provide a finite quasiequational axiomatisation of the equational theory of algebras of binary relations for the intermediate signature of composition, union, and domain. Next we note that our regular sets of trees are not closed under complement, but prove that they are closed under intersection. Finally, we prove that under relational semantics the equational validities of Kleene algebras with domain form a decidable set.
\end{abstract}

\maketitle

\section{Introduction}

Reasoning about binary relations, and ways of combining them, has an extensive literature and a multitude of applications. Classically, in algebraic logic, binary relations model logical formulas with two free variables \cite{Tarski1941}. In computer science, we can find binary relations modelling the actions of programs \cite{pratt1976semantical, kozen1997kleene}, and elsewhere representing relationships between items of data that compose a tree \cite{Benedikt:2009:XL:1456650.1456653}, or a graph \cite{libkin2013querying}.

When an algebraic logician thinks of binary relations, the first signature to come to mind will always be that of Tarski's relation algebras. In computer science, the Kleene algebra signature has the greatest prominence. In the latter case, the operations are relational composition, union, and reflexive transitive closure, as well as constants for the empty relation and the identity relation.

Any set of binary relations closed under the five Kleene algebra operations/\allowbreak const\-ants can be viewed as an algebra in the sense of universal algebra/model theory, that is, a structure over a signature of function symbols (but no predicate symbols). It is well known that this class of algebras contains its free algebras, and that the free algebra generated by a given finite set $\Sigma$ is precisely the set of all regular languages over the alphabet $\Sigma$. The importance of regular languages in theoretical computer science goes almost without saying \cite{regularlanguages}. The algebraic perspective on sets of regular languages was employed to great effect by Eilenberg in his celebrated variety theorem \cite{Eilenberg:1974:ALM:540337}, and continues to yield valuable new insights to this day \cite{gehrke2008duality, GEHRKE20162711}.

In this paper, we identify the analogous free algebras in the case where the signature is expanded with one extra operation on binary relations: the unary \emph{domain} operation
\[\D(R) = \{(x, x)  \mid \exists y  : (x, y) \in R\}\text{,}\]
which provides a record of all points having at least one image under the given relation. This expanded signature is that of \emph{Kleene algebra with domain}, a certain finite set of algebraic laws extending Kozen's theory of Kleene algebras with a domain operation and a small number of associated equations \cite{desharnais2011internal}. 

One intended model for this theory is indeed algebras of binary relations, and there is a hope that the theory will prove useful for reasoning about the actions of nondeterministic computer programs \cite{desharnais2006kleene, DesharnaisS08}. In this programs-as-relations formalism, a program $P$ is modelled by a relation $R$ on machine states, with $R$ relating state $x$ to state $y$ precisely if,  whenever the machine is in state $x$ and $P$ is executed, $y$ is a possible resultant state. Hence relational composition models sequential composition of programs, union models nondeterministic choice, and reflexive transitive closure models the choice to iteratively execute a program any (finite) number of times. The expression $\D(R)$ then models a program that when run from certain states---those from which $P$ would terminate---has no effect---and otherwise fails/does not terminate. Thus domain makes provision within the syntax for expressing certain types of `tests'---an important component of all programming languages.

In addition to identifying the free algebras, we also show that it is decidable whether an equation in this Kleene algebra with domain signature is valid over all algebras of binary relations. Of course, when reasoning about programs, the validity of an equation $s = t$ corresponds to the programs expressed by $s$ and $t$ always having precisely the same effect (independently of which ground programs the variables are instantiated with).

\subsubsection*{Structure of the paper}

 In \Cref{preliminaries} we give the necessary definitions and some context regarding algebras of binary relations, and their free algebras.

 In \Cref{trees} we introduce the trees that we use for describing our free algebras, and certain relations and operations on those trees.

In \Cref{singletons} we prove an intermediate result: we identify the free algebras for the reduced signature that omits the `nondeterministic' union and reflexive transitive closure operators, and also the empty relation constant, that is, the signature with the composition and domain operations and the identity constant. For this signature the elements of the free algebras are `reduced' pointed labelled finite rooted trees (\Cref{main0}).

In \Cref{sets} we extend the result of the previous section to identify the free algebras for the full Kleene algebra with domain signature. In this case, elements of the free algebras are certain sets of the trees of the previous case (\Cref{main}). We term these sets `regular' sets of trees, by analogy with the regular languages of the Kleene algebra signature. By combining with an existing result, it follows as a corollary that the axioms of domain semirings provide a finite quasiequational axiomatisation of the equational theory of algebras of binary relations for the signature of composition, union, domain, and the two constants (but not reflexive transitive closure).

\Cref{automata} is devoted to closure properties of regular sets of trees. We use automata to show the regular sets of trees are closed under intersection. We also note the regular sets of trees are not closed under complement and pose the questions of whether they are closed under implication or under residuation.

In \Cref{section:decide} we again use automata to prove the decidability of validity for equations in the signature of Kleene algebra with domain under relational semantics (\Cref{decidability}).

\section{Algebras of binary relations}\label{preliminaries}

We begin by making precise what is meant by an algebra of binary relations. Throughout, we consider that $0 \in \mathbb N$.

\begin{definition}\label{algebra}
An \defn{algebra of binary relations} of the signature $\{\compo, \bjoin, \kstar, 0, 1\}$ is a universal algebra $\algebra A = (A, \compo, \bjoin, \kstar, 0, 1)$ where the elements of the universe $A$ are  binary relations on some (common) set $X$, the \defn{base}, and the interpretations of the symbols are given as follows:
\begin{itemize}
\item
the binary operation $\compo$ is interpreted as \defn{composition} of relations:
\begin{align*}
R \compo S \coloneqq \{(x, y) \in X^2 \mid \exists z \in X : (x, z) \in R \wedge (z, y) \in S\},
\end{align*}

\item
the binary operation $\bjoin$ is interpreted as set-theoretic \defn{union}:
\begin{align*}R \bjoin S &\coloneqq \{(x, y) \in X^2 \mid (x, y) \in R \vee (x, y) \in S\},\end{align*}

\item
the unary operation $\kstar$ is interpreted as \defn{reflexive transitive closure}:
\begin{align*}R{\kstar} \coloneqq \{&(x, y) \in X^2 \mid \exists n \in \mathbb N\ \ \exists x_0 \ldots x_n :\\
&(x_0 = x) \aand (x_n = y) \aand (x_0 , x_1)\in R \aand \dots \aand (x_{n-1}, x_n) \in R\},\end{align*}

\item
the constant $0$ is interpreted as the \defn{empty} relation:
\begin{align*}0 \coloneqq \emptyset,\end{align*}

\item
the constant $1$ is interpreted as the \defn{identity} relation on $X$:
\begin{align*}1 \coloneqq \{(x, x) \in X^2\}.\end{align*}
\end{itemize}
We let $\operatorname{Rel}(\compo, \bjoin, \kstar, 0, 1)$ denote the isomorphic closure of the class of all algebras of binary relations of the signature $\{\compo, \bjoin, \kstar, 0, 1\}$.
\end{definition}

To be clear: the universe of an algebra of binary relations is necessarily closed under the given operations, since the definition of a universal algebra requires the symbols be interpreted as total functions.

\begin{remark}\leavevmode
\begin{enumerate}[(i)]
\item
It is easy to see that $\operatorname{Rel}(\compo, \bjoin, \kstar, 0, 1)$ is not a first-order axiomatisable class, not even closed under elementary equivalence, by a simple argument showing that $\operatorname{Rel}(\compo, \bjoin, \kstar, 0, 1)$ is not closed under ultrapowers. See the appendix for a proof of this well-known fact.

\item
Despite $\operatorname{Rel}(\compo, \bjoin, \kstar, 0, 1)$ being far from a variety, it is easily seen to be closed under subalgebras and products. (An element of a product of algebras of binary relations is the disjoint union of all its component binary relations.) Hence, by a basic theorem of universal algebra (see, for example, \cite[Theorem 10.12]{burris2011course}), the class $\operatorname{Rel}(\compo, \bjoin, \kstar, 0, 1)$ contains its free algebras.

\item It is a folk theorem that the free $\operatorname{Rel}(\compo, \bjoin, \kstar, 0, 1)$-algebra generated by a finite set $\Sigma$ is the set of all \emph{regular languages} over the alphabet $\Sigma$ (with the  operations of language concatenation, union, and so on).

\item
It is well known that the variety $H S P \operatorname{Rel}(\compo, \bjoin, \kstar, 0, 1)$ generated by $\operatorname{Rel}(\compo,\allowbreak \bjoin,\allowbreak \kstar, 0, 1)$ has no finite equational axiomatisation \cite{redko}.

\item
We do however have Kozen's quasivariety of \emph{Kleene algebras} \cite{kozen1994completeness}, defined by a finite number of equations/quasiequations,\footnote{A \defn{quasiequation} is a conditional equation where the condition is a finite conjunction of equations. That is, a quasiequation is a formula of the form $s_1 = t_1 \wedge \dots \wedge s_n = t_n \limplies u = v$.} intermediate to $\operatorname{Rel}(\compo, \bjoin, \allowbreak\kstar, 0,\allowbreak 1)$ and $H S P \operatorname{Rel}(\compo, \bjoin, \kstar, 0, 1)$. That is, \[\operatorname{Rel}(\compo, \bjoin, \kstar, 0, 1) \subseteq \operatorname{Kleene\ algebras} \subseteq H S P \operatorname{Rel}(\compo, \bjoin, \kstar, 0, 1),\] and so \[H S P (\operatorname{Kleene\ algebras}) = H S P \operatorname{Rel}(\compo, \bjoin, \kstar, 0, 1).\]

\end{enumerate}
\end{remark}

Of course, the operations of \Cref{algebra} are not the only operations that can be defined on binary (endo)relations. In particular, various unary `test' operations can be defined; here is a selection.

\begin{definition}\label{tests}\leavevmode
\begin{itemize}
\item
The unary operation $\D$ is the operation of taking the diagonal of the \defn{domain} of a relation:
\[\D(R) = \{(x, x) \in X^2 \mid \exists y \in X : (x, y) \in R\}\text{.}\]
\item
The unary operation $\R$ is the operation of taking the diagonal of the \defn{range} of a relation:
\[\R(R) = \{(y, y) \in X^2 \mid \exists x \in X : (x, y) \in R\}\text{.}\]
\item

The unary operation $\A$ is the operation of taking the diagonal of the \defn{antidomain} of a relation\textemdash those points of $X$ at which the image of the relation in empty:
\[\A(R) = \{(x, x) \in X^2 \mid \cancel{\exists} y \in X : (x, y) \in R\}\text{.}\]
\end{itemize}
\end{definition}

One can vary the operations from those of \Cref{algebra} and/or restrict the binary relations to some particular form. The resulting class will again contain its free algebras. If the class is of interest, then it is useful to establish a description of these free algebras.

Restricting the binary relations to be some type of function (total functions, partial functions, or injective partial functions, for example) tends to yield free algebras whose elements are a `single object', rather than a `set of objects'. The class of semigroups, for example, is the variety generated by $\operatorname{Tot}(\compo)$---algebras of total functions with composi\-tion---and an element of a free semigroup is a single string, rather than a set of strings as we have in the case $\operatorname{Rel}(\compo, \bjoin, \kstar, 0, 1)$. Similarly, elements of free groups are also strings, with groups forming isomorphs of algebras of \emph{permutations}, with the familiar operations.

There is also an observable pattern when test operations are added to the signature: strings are replaced by (labelled) trees. The following results are known.

\begin{enumerate}
\item
		The class $\operatorname{Inj}(\compo, \inv)$, of isomorphs of algebras of injective partial functions with composition and inverse, is the class of inverse semigroups \cite{wagnergeneralised, doi:10.1112/jlms/s1-29.4.396}.\footnote{In this signature, in which $\inv$ is available, the inverse semigroups form a variety. To give an equational axiomatisation it suffices to replace, in the natural axiomatisation, the quasiequation `inverses are unique' by the equation $a \compo a^{-1} \compo a^{-1} \compo a = a^{-1} \compo a \compo a \compo a^{-1}$ \cite{MR0170966}.} (In this signature $\D$ and $\R$ are definable via $\D(R) \coloneqq R \compo R^{-1}$ and $\R(R) \coloneqq R^{-1} \compo R$, respectively.) Elements of free inverse semigroups are certain trees, so-called \emph{Munn trees} \cite{munn1974free}.

\item
The class $\operatorname{Par}(\compo, \D)$---partial functions with composition and domain---is a variety \cite{Tro73}, most commonly known as the \emph{(left) restriction semigroups}. A description of the free algebras has been given, and again, elements can be viewed as trees~\cite{fountain_1991}.

\item
The class $\operatorname{Par}(\compo, \D, \R)$---partial functions with composition, domain, and range---is a proper quasivariety; a finite quasi\-equational axiomatisation was given by Schein \cite{schein}. Once more, a description of the free algebras has been given, and elements can be viewed as trees \cite{doi:10.1142/S0218196709005214}.
\end{enumerate}

We should also mention at this stage Hollenberg's finite equational axiomatisation of the equational theory of the quasivariety $\operatorname{Rel}(\compo, \A)$ (in which $\D$, $0$, and $1$ are easily expressible) \cite{10230740181599}. Of course, this result amounts to an implicit description of the corresponding free algebras, as quotients of term algebras by this theory. Another result involving binary relations (but not tests), is Bloom, \'{E}sik, and Stefanescu's explicit description of the free algebras for the case $\operatorname{Rel}(\compo, \bjoin, \kstar, 0, 1, {^\smile})$, where $^\smile$ is the converse operation $R^\smile \coloneqq \{(x, y) \in X^2 \mid (y, x) \in R\}$ \cite{Bloom1995}. There, the elements of the algebras are sets of strings.

Having noted that \emph{binary relations} $\leadsto$ \emph{sets}, \emph{functions} $\leadsto$ \emph{singletons}, and \emph{tests} $\leadsto$ \emph{trees}, one can anticipate that when tests are added to the case $\operatorname{Rel}(\compo, \bjoin, \kstar, 0, 1)$, elements of free algebras will be sets of labelled trees. We will prove that this is indeed the case (\Cref{main}). On the way to doing this, we also identify the free algebras for the case without the `nondeterministic' operators $+$ and $\kstar$, more precisely, for the case $\operatorname{Rel}(\compo, 1, \D)$ (\Cref{main0}). The analogous results for signatures formed by adding/removing $0$ and/or $1$ follow as corollaries of these two theorems.

\begin{remark}
The term `Kleene algebra with domain' was originally used for a certain quasiequational theory extending the two-sorted \emph{Kleene algebra with tests} with a domain operation \cite{desharnais2006kleene}. It was subsequently redefined as a (strictly less expressive) one-sorted quasiequational theory extending \emph{Kleene algebra} with a domain operation~\cite{desharnais2011internal}.
\end{remark}

\section{Trees}\label{trees}

The central objects we will be working with throughout will be labelled rooted trees. We will give two definitions of these. The first, \Cref{tree1}, is the usual graph-theoretic definition, and we give it in order to make use of basic graph-theoretic terminology: vertex, edge, and so on. The second, \Cref{tree2}, is cleaner, in the sense that isomorphic trees are identical, and will serve as the `official' definition in this paper.

\begin{terminology}\label{tree1}\leavevmode
\begin{itemize}
\item
A \defn{tree} is a connected acyclic undirected graph (reflexive edges are prohibited). All trees will be assumed to be finite unless otherwise stated.

\item
A \defn{rooted tree} is a tree with a distinguished vertex called the \defn{root}.

\item
By a labelled tree, we will mean an \emph{edge-labelled} tree. That is, given a set $\Sigma$ of labels, a \defn{labelled} tree is a tree $T$ together with a function from the edges of $T$ to $\Sigma$.
\end{itemize}
\end{terminology}

\begin{definition}\label{tree2}
Given a set $\Sigma$ of labels, a \defn{labelled rooted tree} is defined recursively as a set of pairs $(a, T)$, where $a \in \Sigma$ and $T$ is a labelled rooted tree.
\end{definition}

Some explanation may be helpful. According to \Cref{tree2}, the empty set is a labelled rooted tree (this is the base case of the definition). This empty set should be thought of as encoding what is, in the graph-theoretic view, the tree with a single vertex. \Cref{encodings} illustrates how some simple examples of labelled rooted trees should be viewed.

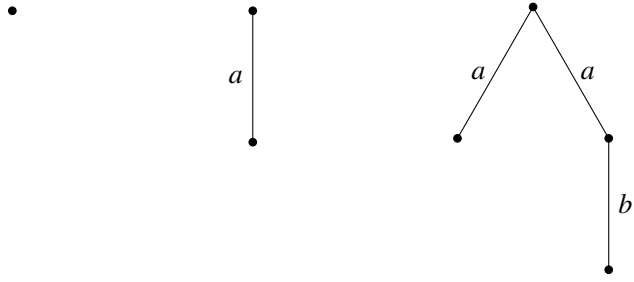
\begin{figure}[H]\centering
\begin{tikzpicture}
\draw[fill] (0,0) circle[radius=.05];
\draw (0, -3.46) ;
\end{tikzpicture}
\hspace{2.5cm}
\begin{tikzpicture}
\draw (0, 0)--node[left]{$a$}(0, -1.73);

\draw[fill] (0,0) circle[radius=.05];
\draw[fill] (0, -1.73) circle[radius=.05];
\draw (0, -3.46) ;
\end{tikzpicture}\hspace{2.5cm}
\begin{tikzpicture}
\draw (0, 0)--node[left]{$a$}(-1, -1.73);
\draw (0, 0)--node[right]{$a$}(1, -1.73);
\draw (1, -1.73)--node[right]{$b$}(1, -3.46);

\draw[fill] (0,0) circle[radius=.05];
\draw[fill] (-1, -1.73) circle[radius=.05];
\draw[fill] (1, -1.73) circle[radius=.05];
\draw[fill] (1, -3.46) circle[radius=.05];
\end{tikzpicture}
\caption{The labelled rooted trees encoded as $\emptyset$, $\{(a, \emptyset)\}$, and $\{(a, \emptyset), (a, \{(b, \emptyset)\})\}$, respectively (with roots at the top)}\label{encodings}
\end{figure}

The reader may note that \Cref{tree2} is more restrictive than the definition of labelled rooted trees obtained from \Cref{tree1}---it cannot describe any tree having a vertex with two distinct but isomorphic child subtrees. However, we will have no need of such trees in this paper.

\begin{definition}
A \defn{pointed tree} is a tree with a distinguished vertex called the \defn{point}.
\end{definition}

We will primarily work with pointed labelled rooted trees. We will usually denote a pointed labelled rooted tree $(T, p)$ by its underlying labelled rooted tree $T$. We define the notion of a homomorphism of (possibly pointed) labelled rooted trees by reference to homomorphisms of relational structures.

\begin{definition}\label{relational_homomorphism}
A \defn{relational structure} $(X, f)$ for a set $\Sigma$ of labels is a set $X$ and an assignment $f$ giving for each element of $a \in \Sigma$ a binary relation $a^f$ on $X$. A \defn{homomorphism} from a relational structure $(X, f)$ to the relational structure $(Y, g)$ (both with label-set $\Sigma$) is a map $\theta \from X \to Y$ validating  
\[(x_1, x_2) \in a^f \implies (\theta(x_1), \theta(x_2)) \in a^g\] for each $a \in \Sigma$.
\end{definition}

We can view any $\Sigma$-labelled rooted tree as a relational structure $(X, f)$ by taking $X$ to be the set of vertices of the tree and for each $a \in \Sigma$ setting $a^f$ to be the set of pairs $(x, y)$ of vertices such that $x$ is the parent of $y$ and the edge $\{x, y\}$ is labelled by $a$. When we speak of a homomorphism $\theta \from S \to T$ of possibly pointed labelled rooted trees, we mean a homomorphism of the trees viewed as relational structures that is also required to map the root of $S$ to the root of $T$ and, if it exists, the point of $S$ to the (therefore extant) point of $T$.

Let $\algebra A$ be an algebra of binary relations (in any of the signatures we take an interest in) with base set $X$. Let $f$ be an assignment of members of $\algebra A$ to the set $\Sigma$ of variables. Then $(\algebra A, f)$ naturally defines a relational structure: $(X, f)$. Conversely, let $(X, f)$ be a relational structure and let $\algebra A$ be any algebra of binary relations on $X$ that includes $\{a^f \mid a \in \Sigma\}$ in its universe (that is, any algebra between the algebra generated by this set and the algebra of \emph{all} binary relations on $X \times X$). Then the standard model-theoretic interpretation $\interp t^{\algebra A, f}$ of any $\Sigma$-term $t$ is independent of the precise choice of $\algebra A$. Thus, when interpreting terms and evaluating equations, it is safe to conflate the concepts of \emph{algebra + assignment} and \emph{relational structure}, and we will often do so. For example, the following definition is stated in terms of an algebra and assignment, but we will mainly use it in contexts where we are ostensibly talking about a relational structure.

\begin{definition}\label{satisfaction}
Let $\algebra A$ be an algebra of binary relations with base $X$, let $t$ be a term, and let $f$ be an assignment to the variables in $t$. We say that a pair $(x, y) \in X \times X$ \defn{satisfies} $t$ if $(x, y) \in \interp t^{\algebra A, f}$.
\end{definition}

We want to be able to reduce trees to forms without any redundant branches. In order to do that, we first define a preorder on trees.

\begin{definition}\label{ordering}
The preorder $\leq$ on (possibly pointed) labelled rooted trees is defined recursively as follows. For trees $T_1$ and $T_2$ with roots $r_1$ and $r_2$ respectively, $T_1 \leq T_2$ if and only if 
\begin{enumerate}[(a)]
\item\label{a}
if $r_2$ is the point vertex of $T_2$, then $r_1$ is the point vertex of $T_1$,

\item\label{b}
for each child $v_2$ of $r_2$, there is a child $v_1$ of $r_1$ such that 
\begin{enumerate}[(i)]
\item
the labels of the edges $r_1 v_1$ and $r_2 v_2$ are equal,

\item
$T_{v_1} \leq T_{v_2}$, where $T_{v_1} $ and $ T_{v_2}$ are the $v_1$-rooted and $v_2$-rooted subtrees respectively.
\end{enumerate}
\end{enumerate}
\end{definition}

That $\leq$ is indeed a preorder is clear. In fact, by induction on the height of the trees, it is easy to see that
\begin{equation}\label{homomorphism}T_1 \leq T_2 \iff \text{there exists a homomorphism }\theta \from T_2 \to T_1.\end{equation}

In the following definition and proposition, we continue to work with trees that may or may not be pointed.

\begin{definition}\label{def:reduction}
Let $T$ be a labelled tree with root $r$. The \defn{reduced form} of $T$ is the tree formed recursively as follows. 
\begin{enumerate}[(a)]
\item\label{first-step}
For each child $v$ of $r$, replace the $v$-rooted subtree with its reduced form. Call the resulting tree $T'$.
\item
For each label $a \in \Sigma$, let $C_a$ be the set of subtrees of $T'$ rooted at a vertex linked to $r$ by an $a$-labelled edge. Remove all but the $\leq$-minimal subtrees in $C_a$.
\end{enumerate}
\end{definition}

\begin{example}[reduction]
The pointed tree on the left of \Cref{fig2} reduces as shown. The tree on the right is already reduced.
\begin{figure}[H]\centering
\begin{tikzpicture}
\draw (0, 0)--node[left]{$a$}(-1, -1.73);
\draw (0, 0)--node[right]{$a$}(1, -1.73);
\draw (-1, -1.73)--node[left]{$a$}(-2, 2*-1.73);
\draw (-1, -1.73)--node[right]{$b$}(0, 2*-1.73);

\draw[fill] (0,0) circle[radius=.05];
\draw[fill] (-1, -1.73) circle[radius=.05];
\draw (-1, -1.73) circle[radius=.1];
\draw[fill] (1, -1.73) circle[radius=.05];
\draw[fill] (-2, 2*-1.73) circle[radius=.05];
\draw[fill] (0, 2*-1.73) circle[radius=.05];
\end{tikzpicture}\hspace{.65cm}\scalebox{2}{$\overset{\leadsto}{\phantom{\framebox(0,25){}}}$}\hspace{-.35cm}
\begin{tikzpicture}
\draw (0, 0)--node[left]{$a$}(-1, -1.73);
\draw (-1, -1.73)--node[left]{$a$}(-2, 2*-1.73);
\draw (-1, -1.73)--node[right]{$b$}(0, 2*-1.73);

\draw[fill] (0,0) circle[radius=.05];
\draw[fill] (-1, -1.73) circle[radius=.05];
\draw (-1, -1.73) circle[radius=.1];
\draw[fill] (-2, 2*-1.73) circle[radius=.05];
\draw[fill] (0, 2*-1.73) circle[radius=.05];
\end{tikzpicture}\hspace{2.1cm}
\begin{tikzpicture}
\draw (0, 0)--node[left]{$a$}(-1, -1.73);
\draw (0, 0)--node[right]{$a$}(1, -1.73);
\draw (1, -1.73)--node[right]{$b$}(2, 2*-1.73);
\draw (1, -1.73)--node[left]{$a$}(0, 2*-1.73);

\draw[fill] (0,0) circle[radius=.05];
\draw[fill] (-1, -1.73) circle[radius=.05];
\draw[fill] (1, -1.73) circle[radius=.05];
\draw (-1, -1.73) circle[radius=.1];
\draw[fill] (2, 2*-1.73) circle[radius=.05];
\draw[fill] (0, 2*-1.73) circle[radius=.05];
\end{tikzpicture}

\caption{Reduction of pointed labelled rooted trees}\label{fig2}
\end{figure}
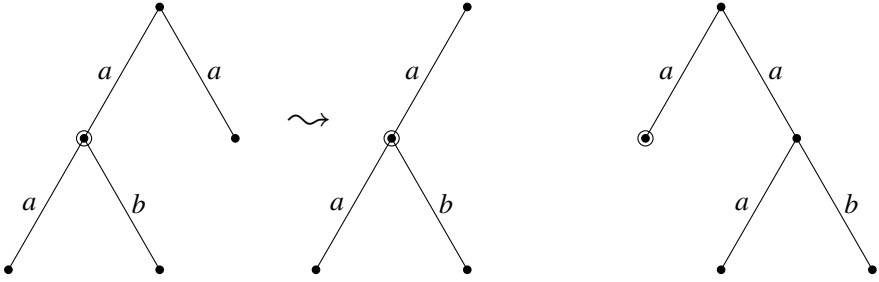
\end{example}

\begin{proposition}\label{order}
The preorder $\leq$ is a partial order on reduced labelled rooted trees.
\end{proposition}

\begin{proof}
By induction on the maximum height of the two trees being compared. For the base case, take trees $T_1, T_2$ of height $0$, so both have just a single vertex. Suppose $T_1 \leq T_2$ and $T_2 \leq T_1$. Then by \Cref{ordering}\ref{a}, the tree $T_1$ is pointed if and only if $T_2$ is pointed. Hence $T_1 = T_2$.

Now let $T_1$ and $T_2$ be of height at most $n + 1$, and assume antisymmetry of $\leq$ holds for all reduced trees of height at most $n$. Suppose $T_1 \leq T_2$ and $T_2 \leq T_1$, and denote the roots by $r_1$ and $r_2$, respectively. As before, $T_1$ has a point at $r_1$ if and only if $T_2$ has a point at $r_2$. By \Cref{def:reduction} the child subtrees of vertex $r_1$ are reduced, and for each $a \in \Sigma$ the child subtrees linked by an $a$-labelled edge are pairwise $\leq$-incomparable. Likewise for the child subtrees of vertex $r_2$.

Let $v_2$ be an arbitrary child of $r_2$. Since $T_1 \leq T_2$, we can find a child $v_1$ of $r_1$ as in \Cref{ordering}\ref{b}, giving us $T_{v_1} \leq T_{v_2}$. Then using $T_2 \leq T_1$ and applying \Cref{ordering}\ref{b} again, we obtain a child $v_2'$ of $r_2$ and have $T_{v_2'} \leq T_{v_1} \leq T_{v_2}$ (with $r_2 v_2$ and $r_2 v_2'$ having the same label). Hence by pairwise $\leq$-incomparability, $T_{v_2} = T_{v_2'}$. So $T_{v_2} \leq T_{v_1} \leq T_{v_2}$, which by the inductive hypothesis yields $T_{v_1} = T_{v_2}$. Since $v_2$ was arbitrary, we conclude that every child subtree of $r_2$ is present as a child subtree of $r_1$. Symmetrically, every child subtree of $r_1$ is present as a child subtree of $r_2$. Hence $T_1 = T_2$.
\end{proof}

Note now that the partial order $\leq$ on reduced trees is Noetherian (converse well-founded). Indeed, by structural induction there are a finite number of distinct labelled trees of any fixed depth, hence a finite number of reduced pointed labelled trees of that depth. And $T_1 \leq T_2$ implies the depth of $T_2$ is at most the depth of $T_1$.

\begin{proposition}\label{order2}
If a labelled rooted tree $T$ reduces to $T'$, then $T \leq T'$ and $T' \leq T$.
\end{proposition}

\begin{proof}
Another induction on height. All trees of height $0$ are already reduced, so this base case is trivial.

Now assume the result holds whenever $T$ is of height at most $n$. Let $T$ be of height $n+1$, with root $r$. It is clear that the root $r'$ of $T'$ is the point of $T'$ if and only if $r$ is the point of $T$. So there is no obstruction to a homomorphism mapping $r$ to $r'$ or vice versa.

First we show $T \leq T'$. Every child subtree of $r'$ is the reduced form $T'_{v'}$ of some child subtree $T_v$ of $r$ (with $rv$ and $r' v'$ having the same label). Hence, by the inductive hypothesis and \eqref{homomorphism}, there is a homomorphism $\theta_{v'} \from T'_{v'} \to T_v$. By gluing together $\{\theta_{v'} \mid v' \text{ a child of $r'$}\}$, and mapping $r'$ to $r$, we obtain a homomorphism $T' \to T$. Hence $T \leq T'$.

Now we show $T' \leq T$. Let $T''$ be the intermediate tree in the reduction of $T$ to $T'$, that is, $T''$ is the tree formed from $T$ by reducing all child subtrees of $r$. Denote the root of $T''$ by $r''$. Take an arbitrary child subtree $T_v$ of $r$, and let $T_{v''}$ be its reduced form sitting as a subtree of $T''$. So by the inductive hypothesis, there is a homomorphism $T_v \to T_{v''}$. Since $\leq$ is Noetherian on reduced trees, when $T'$ is formed from $T''$ by applying step \ref{b} in the definition of reduction, a subtree $T_{v'}$ with $T_{v'} \leq T_{v''}$ (and such that $r'' v''$ and $r' v'$ have the same label) is retained. Thus there is a homomorphism $T_{v''} \to T_{v'}$, and hence, by composition, a homomorphism $T_v \to T_{v'}$. Gluing together the homomorphisms for each child $v$ of $r$, and mapping $r$ to $r'$, we obtain a homomorphism $T \to T'$. Hence $T' \leq T$.
\end{proof}

Thus reduction selects a canonical member of every $\leq$-equivalence class.  Informally, we can think of $\leq$ on reduced trees as corresponding to the \emph{inclusion relation} on binary relations---if $T_1 \leq T_2$ then $T_1$ is a \emph{more specific} description than $T_2$.

\begin{definition}\label{operations}
Let $\Sigma$ be a set and let $T$ and $S$ be reduced pointed $\Sigma$-labelled rooted trees.\begin{itemize}
\item
 The \defn{pointed tree concatenation} $T \compo S$ of $T$ and $S$ is the tree formed by
 \begin{enumerate}
\item
identifying the \emph{point} of $T$ and the \emph{root} of $S$ (the root is now the root of $T$ and the point is the point of $S$),

\item reducing the resulting tree to its reduced form.
\end{enumerate}

\item 
The \defn{domain} $\D(T)$ of $T$ is the tree formed by
\begin{enumerate}
\item
 \emph{reassigning} the point of $T$ to the current root of $T$,
 \item
 reducing the resulting tree to its reduced form.
\end{enumerate}

\end{itemize}
\end{definition}

\begin{notation}
For a symbol $a$ from an alphabet, we write $\vec a$ for the pointed labelled rooted tree with two vertices, whose point is the child vertex and whose single edge is labelled by $a$. We write $\trivial$ for the pointed labelled rooted tree with a single vertex.
\end{notation}

\begin{remark}
A very similar setup to that presented in this section has already been used for investigating Kleene algebras with domain \cite{mbacke2018completeness}. In that thesis, the graph-theoretic definition of trees is used, and pointed labelled finite rooted trees are called `trees with a top'. There, the relation $\leq$ is termed `simulates', and trees are only considered up to simulation equivalence. Thus there is no notion of a reduced form; the operations of \Cref{operations} are defined without their reduction steps.

We will return to say more about the thesis \cite{mbacke2018completeness} at the end of \Cref{sets}.
\end{remark}

\section{Composition, identity, and domain}\label{singletons}
 
In this section we will identify the free algebras of the class $\operatorname{Rel}(\compo, 1, \D)$. From there, it is straightforward to accommodate the addition of $\bjoin$ and $\kstar$ (and $0$). 

By a \emph{term} we mean a raw syntactic object belonging to a term algebra/abso\-lu\-tely free algebra---no background theory is assumed. Thus equality of terms means literal equality.

\begin{definition}\label{recursive}
We define the \defn{single-tree interpretation} $\tinterp{\,\cdot\,}$ of $\{\compo,\allowbreak 1, \D\}$-terms as follows.
\begin{enumerate}
\item
$\tinterp a \coloneqq  \vec a $, for any variable $a$,

\item
$\tinterp 1 \coloneqq \trivial$,

\item
$\tinterp{s \compo t} \coloneqq \tinterp s \compo \tinterp t$,

\item
$\tinterp {\D(s)} \coloneqq \D (\tinterp s)$.
\end{enumerate}
Here the operations and constants on the right-hand side are those defined for trees in \Cref{trees}.
\end{definition}

\Cref{recursive} is well defined because if pressed for a formal definition of terms we would give one that satisfies unique readability, either by writing $\compo$ in prefix form, requiring parentheses, or defining terms directly as trees.

\begin{lemma}\label{observation}
The map $\tinterp{\,\cdot\,}$ is a surjection from the $\{\compo,\allowbreak 1, \D\}$-terms in variables $\Sigma$ onto the reduced pointed $\Sigma$-labelled rooted trees. Hence $\tinterp{\,\cdot\,}$ is a surjective homomorphism from the term algebra onto the reduced pointed $\Sigma$-labelled rooted trees, viewed as an algebra of the signature $\{\compo,\allowbreak 1, \D\}$ with the operations we have defined.
\end{lemma} 

\begin{proof}
We show that every reduced tree is the single-tree interpretation of some term by induction on the size of the tree.

A tree whose root has no children is just $\trivial$---the interpretation of $1$---so assume we have a  tree $T$ given by a nonempty set $\{(a_1, T_1), \dots, (a_n, T_n)\}$, and a distinguished point $p$. First suppose $p$ is the root of $T$. For each $1 \leq i \leq n$, let $t_i$ be a term whose interpretation is the pointed tree whose tree is $T_i$ and whose point is its root. Then we can realise $T$ by the term $\D(a_1 \compo t_1) \compo \dots \compo \D(a_n \compo t_n)$. (We write iterated $\compo$ without brackets because the positioning of the brackets turns out to be immaterial.)

Alternatively, suppose $p$ is not the root of $T$; so without loss of generality $p$ is a vertex in $T_n$. Let $t_1, \dots, t_{n-1}$ be as before and now let $t_n$ be a term whose interpretation is the pointed tree whose tree is $T_n$ and whose point is $p$. Then we can realise $T$ by the term $\D(a_1 \compo t_1) \compo \dots \compo \D(a_{n-1} \compo t_{n-1}) \compo a_n \compo t_n$. 
\end{proof}

The following lemma says that the single-tree interpretation of a term records for us a tree that, in a relational structure, `connects' $x$ and $y$ if and only if the term is satisfied by $(x, y)$.

\begin{lemma}\label{path0}
Let $\algebra A$ be a $\{\compo, 1, \D\}$-algebra of binary relations, with base $X$. Let $t$ be a $\{\compo, 1, \D\}$-term, and let $f$ be an assignment of elements of $\algebra A$---so, binary relations---to the variables in $t$. Then for any $x, y \in X$, the following are equivalent.
\begin{enumerate}
\item
The pair $(x, y)$ belongs to the model-theoretic interpretation of $t$ under the assignment $f$.

\item
There is a homomorphism of $\tinterp t$ into $(X, f)$ such that the root of $\tinterp t$ is mapped to $x$ and the point of $\tinterp t$ is mapped to $y$.
\end{enumerate}
\end{lemma}

\begin{proof}
By structural induction on terms. As before, we write $\interp t^{\algebra A, f}$ for the model-theoretic interpretation of $t$. 

First, suppose $t = a$, for some variable $a$. Then $\tinterp t = \vec a$---consisting of an $a$-labelled edge linking the root $r$ to the point $p$. Therefore $(x, y) \in \interp t^{\algebra A, f}$ if and only if  $(x, y) \in f(a)$, so if and only if the map given by $r \mapsto x$ and $p \mapsto y$ is a homomorphism. The case $t = 1$ is similar.

Next, suppose $t = s_1 \compo s_2$ and that the equivalence holds for $s_1$ and for $s_2$. Then $(x, y) \in \interp t^{\algebra A, f}$ if and only if there is a $z \in X$ such that $(x, z) \in \interp {s_1}^{\algebra A, f}$ and $(z, y) \in \interp {s_2}^{\algebra A, f}$. By the inductive hypotheses, the latter is equivalent to the existence of a homomorphism from $\tinterp {s_1}$ mapping its root $r_1$ to $x$ and point $p_1$ to $z$, and a homomorphism from $\tinterp {s_2}$ mapping its root $r_2$ to $z$ and point $p_2$ to $y$. This is equivalent to the existence of a homomorphism from $\tinterp {s_1 \compo s_2}$, mapping the root to $x$ and point to $y$, since there exist homomorphisms in both directions between a tree and its reduced form.

Finally, suppose $t = \D(s)$ and that the equivalence holds for $s$. Then $(x, y) \in \interp t^{\algebra A, f}$ if and only if $x = y$ and there exists $z$ such that $(x, z) \in \interp s^{\algebra A, f}$. By the inductive hypothesis, this is equivalent to $x = y$ and the existence of a homomorphism mapping the root of $\tinterp s$ to $x$ and the point to $z$. Since $\tinterp t$ and $\tinterp s$ differ only by the position of their point, which for $\tinterp t$ is the root, the latest condition is equivalent to the existence of a homomorphism from $\tinterp t$ mapping both the root/point to $x$/$y$, as required.
\end{proof}

\begin{corollary}[soundness with respect to relations]\label{sound0}
For any pair $s$ and $t$ of $\{\compo, 1, \D\}$-terms 
\[\tinterp s = \tinterp t \implies \operatorname{Rel}(\compo, 1, \D) \models s = t.\]
\end{corollary}

\begin{lemma}[completeness with respect to relations]\label{complete0}
For any pair $s$ and $t$ of $\{\compo, 1, \D\}$-terms 
\[\operatorname{Rel}(\compo, 1, \D) \models s = t \implies \tinterp s = \tinterp t.\]
\end{lemma}

\begin{proof}
Suppose $\operatorname{Rel}(\compo, 1, \D) \models s = t$. We will show that $\tinterp s \leq \tinterp t$. Then by symmetry, also $\tinterp t \leq \tinterp s$. So $\tinterp s = \tinterp t$ and we are done.

Let $r$ be the root and $p$ the point of $\tinterp s$. On $\tinterp s$ viewed as a relational structure,  by using \Cref{path0} and the identity homomorphism $\tinterp s \to \tinterp s$, the pair $(r, p)$ belongs to the model-theoretic interpretation of $s$ (under the evident variable assignment). Hence, by the assumption that $\operatorname{Rel}(\compo, 1, \D) \models s = t$, we have that $(r, p)$ belongs to the model-theoretic interpretation of $t$. Then by \Cref{path0} there is a homomorphism of labelled trees from $\tinterp t$ into $\tinterp s$ mapping the root of $\tinterp t$ to $r$ and the point of $\tinterp t$ to $p$. That is, there is a homomorphism of pointed labelled rooted trees from $\tinterp t$ into $\tinterp s$. We know this is equivalent to the conclusion $\tinterp s \leq \tinterp t$ we seek, so we are done.\end{proof}

With \Cref{sound0} and \Cref{complete0} available, we can now complete the proof that we have identified the free algebras of the class $\operatorname{Rel}(\compo, 1, \D)$.

\begin{theorem}\label{main0}
Let $\Sigma$ be an alphabet, and let $\mathcal R_\Sigma$ be the set of reduced pointed $\Sigma$-labelled finite rooted trees. Then the free $\operatorname{Rel}(\compo, 1, \D)\text{-algebra}$ over $\Sigma$ is $\mathcal R_\Sigma$ equipped with the operations of pointed tree concatenation and of domain from \Cref{operations} (and the constant $\trivial$).
\end{theorem}

\begin{proof}
Since $\mathcal R_\Sigma$ is generated by $\Sigma$, it follows by the first isomorphism theorem of universal algebra that $\mathcal R_\Sigma$ is isomorphic to a quotient $\algebra Q$ of the term algebra (for signature $\{\compo, 1, \D\}$) over variables $\Sigma$. The congruence relation $\sim$ defining the quotient is given by $s \sim t \iff \tinterp s = \tinterp t$, for terms $s$ and $t$ over $\Sigma$. So by \Cref{sound0} and \Cref{complete0}, the congruence $\sim$ is given by equational validity in $\operatorname{Rel}(\compo, 1, \D)$. It is a basic result of universal algebra that $\algebra Q$ is then precisely the free algebra over $\Sigma$ of the class $\operatorname{Rel}(\compo, 1, \D)$.
\end{proof}

The inclusion of $1$ and exclusion of $0$ from the signature was in fact not essential. We can easily add and remove them according to our wishes.

\begin{corollary}\label{constants}
Let $\Sigma$ and $\mathcal R_\Sigma$ (viewed as an algebra) be as in \Cref{main0}.
\begin{itemize}
\item
The free $\operatorname{Rel}(\compo, \D)$-algebra over $\Sigma$ is given by removing the tree $\trivial$ from $\mathcal R_\Sigma$.

\item
The free $\operatorname{Rel}(\compo, 0, \D)$-algebra over $\Sigma$ is given by the addition of a zero element---an element validating $0 \compo T  = T \compo 0 = 0$ and $\D(0) = 0$---to the free $\operatorname{Rel}(\compo, \D)$-algebra over $\Sigma$.

\item
The free $\operatorname{Rel}(\compo, 1, 0, \D)$-algebra over $\Sigma$ is given by the addition of a zero element to $\mathcal R_\Sigma$.
\end{itemize}
\end{corollary}

\begin{proof}
For $\operatorname{Rel}(\compo, \D)$, first note that every nontrivial tree in $\mathcal R_\Sigma$ is the single-tree interpretation of a $\{\compo, \D\}$-term, and conversely, if a tree is the interpretation of a $\{\compo, \D\}$-term then it cannot be trivial---by induction, all such interpretations have at least one edge. Hence the nontrivial trees indeed form a $\{\compo, \D\}$-algebra generated by (the interpretations of elements of) $\Sigma$. Since every $\{\compo, \D\}$-algebra of relations embeds in a $\{\compo, 1, \D\}$-algebra of relations, it follows from \Cref{sound0} that every equation validated by the nontrivial trees is validated by all $\{\compo, \D\}$-algebras of relations. The converse follows immediately from \Cref{complete0}.

For $\operatorname{Rel}(\compo, 0, \D)$ and $\operatorname{Rel}(\compo, 1, 0, \D)$, note that by the definition of a zero element, a term is interpreted as the zero of the algebra if and only if the symbol $0$ appears in the term, and similarly a term is interpreted as $\emptyset$ in every algebra of relations if and only if $0$ appears in the term. These observations are sufficient to extend \Cref{sound0} and \Cref{complete0} to terms that may contain $0$.
\end{proof}

We remark that the class $\operatorname{Rel}(\compo, 1, \D)$, as noted in \cite{HIRSCH201175}, forms a quasivariety.\footnote{The short proof of this using general model-theoretic results consists of noting that the class is both closed under direct products and---almost by definition---has a pseudouniversal axiomatisation. (See \cite[Section 9.2]{bygames} for a definition of \emph{pseudouniversal}.)} It has been shown that this quasivariety is not finitely axiomatisable in first-order logic \cite{HIRSCH201175}. Naturally, the same statements hold when we add/remove $0$ and $1$.

\section{Expansion by union and reflexive transitive closure}\label{sets}

In this section, we extend the result of the previous section to provide a description of the free algebras of the class of relational Kleene algebras with domain, that is, the free algebras of $\operatorname{Rel}(\compo, \bjoin, \kstar, 0, 1, \D)$.

\begin{definition}\label{standard}
For a set $K$ of reduced trees, let $\operatorname{maximal} (K)$ denote the set of $\leq$-maximal elements of $K$. We lift the notation $\compo$ and $\D$ of \Cref{operations} to sets of trees by using elementwise application. We define the \defn{standard tree interpretation} $\llbracket \, \cdot \, \rrbracket$ of $\{\compo, \bjoin, \kstar, 0,\allowbreak 1, \D\}$-terms as follows.
\begin{enumerate}
\item
For $a \in \Sigma, \ \interp a \coloneqq \{ \vec a \}$,
\item
$\interp 0 \coloneqq \emptyset$,

\item
$\interp 1 \coloneqq \{\trivial\}$,

\item
$\interp{ s \bjoin t} \coloneqq \operatorname{maximal}(\interp s \cup \interp t)$,
\item
$\interp{s \compo t} \coloneqq \operatorname{maximal}(\interp{s} \compo \interp{t} )$,
\item
$\interp s^* \coloneqq \operatorname{maximal}(\bigcup_{i = 1}^{\infty} \interp s^i)$, where $\interp s^0 \coloneqq 1$ and $\interp s^{i+1} \coloneqq \interp s^{i} \compo \interp s$,

\item
$\interp {\D(s)} \coloneqq \operatorname{maximal}(\D \interp s )$.
\end{enumerate}
\end{definition}

Note that $\interp a$, for $a \in \Sigma$, $\interp 0$, and $\interp 1$, contain only reduced trees, and  $\cup$ preserves this property on sets of trees (as do the lifted $\compo$ and $\D$, by definition). Hence the maximal operation is applicable whenever it is used in \Cref{standard}, and standard interpretations contain only reduced trees.

\begin{definition}
Let $\Sigma$ be an alphabet. A set of pointed $\Sigma$-labelled rooted trees is \defn{regular} if it is the standard tree interpretation of some $\{\compo, \bjoin, \kstar, 0, 1, \D\}$-term.\footnote{It is not claimed that this notion of regular sets of \emph{pointed} trees is the same as the notion of a \emph{regular tree language} coming from the theory of tree automata \cite{handbookformallanguages}.}
\end{definition}

We can think of a regular set $L$ of trees as a concise record of all the reduced trees in the downward-closed set $\down L$ (with respect to the $\leq$ ordering). In this view (thinking of $L$ as $\down L$), the operation $+$ corresponds to the real set union operation, and $\compo$ and $\D$ correspond to pointwise application of the operations of \Cref{operations}. The advantage of using the arrangement of \Cref{standard} is that regular sets remain \emph{finite} until such time that Kleene star is used. That the partial order $\leq$ on reduced trees is Noetherian ensures there are `enough' elements in $\operatorname{maximal}(\bigcup_{i = 1}^{\infty} \interp s^i)$ for that to be a sensible definition of ${\interp s}^*$ (as demonstrated in the proof of the next lemma).

The following lemma is the analogue of \Cref{path0}.

\begin{lemma}\label{path}
Let $\algebra A$ be a $\{\compo, \bjoin, \kstar, 0, 1, \D\}$-algebra of binary relations, with base $X$. Let $t$ be a $\{\compo, \bjoin, \kstar, 0, 1, \D\}$-term, and let $f$ be an assignment of elements of $\algebra A$ to the variables in $t$. Then for any $x, y \in X$, the following are equivalent.
\begin{enumerate}
\item\label{interpretation}
The pair $(x, y)$ belongs to the model-theoretic interpretation of $t$ under the assignment $f$.

\item\label{there_exists}
There is a tree $T$ in $\interp t$ and a homomorphism of $T$ into $(X, f)$ such that the root of $T$ is mapped to $x$ and the point of $T$ is mapped to $y$.
\end{enumerate}
\end{lemma}

\begin{proof}
Structural induction on terms. We give the details for the $\kstar$ case in the direction \ref{interpretation} $\implies$ \ref{there_exists}, as this case is not entirely trivial. So suppose that $t = s^*$, that condition \ref{interpretation} holds for $t$, and that whenever condition \ref{interpretation} holds for $s$, condition \ref{there_exists} holds for $s$. Then since $(x, y)$ belongs to the model-theoretic interpretation of $s^*$, there are $z_0, \dots, z_n$ with $z_0 = x$, $z_n = y$, and each $(z_i, z_{i+1})$ belonging to the interpretation of $s$. Hence there are trees $S_1, \dots, S_n \in \interp s$ and homomorphisms mapping each $S_i$ into $(X, f)$ with the root mapping to $z_{i-1}$ and point mapping to $z_i$. By \eqref{homomorphism}, there is a homomorphism of the reduced tree $(\dots(S_1 \compo S_2) \compo \dots \compo S_n)$ into $(X, f)$ mapping the root to $x$ and the point to $y$. By \eqref{homomorphism} again (and an induction up to $n$), there is a $T' \in \interp{s}^n$ with $T' \geq (\dots(S_1 \compo S_2) \compo \dots \compo S_n)$ and a homomorphism of $T'$ into $(X, f)$ mapping the root to $x$ and the point to $y$. Then $T' \in \bigcup_{j = 1}^{\infty} \interp s^j$, and since the partial order $\leq$ on reduced trees is Noetherian, there exists a $T \in \interp s^* \coloneqq \operatorname{maximal}(\bigcup_{i = 1}^{\infty} \interp s^i)$ with $T \geq T'$. Such a $T$ fulfils condition \ref{there_exists}, so we are done.
\end{proof}

\begin{proposition}[soundness with respect to relations]\label{sound}
For any pair $s$ and $t$ of $\{\compo, \bjoin, \kstar, 0, 1, \D\}$-terms 
\[\interp s = \interp t \implies \operatorname{Rel}(\compo, \bjoin, \kstar, 0, 1, \D) \models s = t.\]
\end{proposition}

\begin{proof}
Suppose $\interp s = \interp t$. Let $\algebra A$ be a $\{\compo, \bjoin, \kstar, 0, 1, \D\}$-algebra of binary relations, with base $X$. Let $f$ be an assignment of elements of $\algebra A$ to the variables appearing in $s = t$. Write $\interp {\, \cdot\, }^{\algebra A, f}$ for the model-theoretic interpretations in $\algebra A$ under $f$. Then by \Cref{path}, for any $x, y \in X$, we have that $(x, y) \in \interp {s}^{\algebra A, f}$ if and only if there is a $T \in \interp s$ with $T$ connecting $x$ and $y$. As $\interp s = \interp t$, this is equivalent to there being a $T \in \interp t$ with $T$ connecting $x$ and $y$, which in turn is equivalent, by \Cref{path} again, to having $(x, y) \in \interp {t}^{\algebra A, f}$. As $x$ and $y$ were arbitrary, we have $\interp {s}^{\algebra A, f} = \interp {t}^{\algebra A, f}$. As $\algebra A$ and $f$ were arbitrary, we conclude that $\operatorname{Rel}(\compo, \bjoin, \kstar, 0, 1, \D) \models s = t$.
\end{proof}

\begin{proposition}[completeness with respect to relations]\label{complete}
For any pair $s$ and $t$ of $\{\compo, \bjoin, \kstar, 0, 1, \D\}$-terms 
\[\operatorname{Rel}(\compo, \bjoin, \kstar, 0, 1, \D) \models s = t \implies \interp s = \interp t.\]
\end{proposition}

\begin{proof}
Just like the proof of \Cref{complete0}. Given $S \in \interp s$, we obtain the existence of a $T \in \interp t$ with $S \leq T$. By symmetry, there is an $S' \in \interp s$ with $T \leq S'$. Since $\interp s$ is a $\leq$-antichain, $S = T$. Hence $\interp s \subseteq \interp t$. By symmetry, also $\interp t \subseteq \interp s$.
\end{proof}

With \Cref{sound} and \Cref{complete}, we can now complete the proof of \Cref{main} in the familiar way.

\begin{theorem}\label{main}
Let $\Sigma$ be an alphabet, and let $\mathcal R_\Sigma$ be the set of reduced pointed $\Sigma$-labelled rooted trees. Then the free $ \operatorname{Rel}(\compo, \bjoin, \kstar, 0, 1, \D)$-algebra over $\Sigma$ has as its universe all the regular subsets of $\mathcal R_\Sigma$. The operations are the following, where $L$, $L_1$, and $L_2$, are regular sets of reduced trees.
\begin{enumerate}
\item
$ 0 \coloneqq \emptyset$,

\item
$ 1 \coloneqq \{\trivial\}$,

\item
$L_1 \bjoin L_2 \coloneqq \operatorname{maximal}(L_1 \cup L_2)$,
\item
$L_1 \compo L_2 \coloneqq \operatorname{maximal}( L_1 \compo L_2)$,
\item
$L^* \coloneqq \operatorname{maximal}(\bigcup_{i = 1}^{\infty} L^i)$, where $L^0 \coloneqq \{\trivial\}$ and $L^{i+1} \coloneqq L^{i} \compo L$,

\item
$\D(L) \coloneqq \operatorname{maximal}( \D (L) )$.
\end{enumerate}
\end{theorem}

\begin{proof}
Similar to the proof of \Cref{main0}.
\end{proof}

\begin{corollary}
Let $\Sigma$ and $\mathcal R_\Sigma$ be as in \Cref{main}.

\begin{itemize}
\item
The free $\operatorname{Rel}(\compo, \bjoin, \kstar, 0, \D)$-algebra over $\Sigma$ consists of the regular subsets of $\mathcal R_\Sigma$ that do not contain the trivial tree $\trivial$.

\item
The free $\operatorname{Rel}(\compo, \bjoin, \kstar, 1, \D)$-algebra over $\Sigma$ consists of the nonempty regular subsets of $\mathcal R_\Sigma$.

\item
The free $\operatorname{Rel}(\compo, \bjoin, \kstar, \D)$-algebra over $\Sigma$ consists of the nonempty regular subsets of $\mathcal R_\Sigma$ that do not contain $\trivial$.
\end{itemize}
\end{corollary}

\begin{proof}
Similar to the proof of \Cref{constants}.
\end{proof}

\begin{corollary}\label{finite}
Let $\Sigma$ and $\mathcal R_\Sigma$ be as in \Cref{main}. Then the free $\operatorname{Rel}(\compo, \bjoin, 0, 1,\allowbreak \D)$-algebra over $\Sigma$ consists of all finite regular subsets of $\mathcal R_\Sigma$.
\end{corollary}

\begin{proof}
First note that the regular subsets interpreting $\{\compo, \bjoin, 0, 1, \D\}$-terms are precisely the finite regular subsets. Then soundness and completeness follow from \Cref{sound} and \Cref{complete} respectively.
\end{proof}

In \cite{mbacke2018completeness}, Mbacke proves (Theorem 5.3.3 there) that a certain finite equational theory over the signature $\{\compo, \bjoin, 0, 1, \D\}$---the theory of \emph{domain semirings}---is complete for the equational validities of what amounts to the algebras of trees identified in \Cref{finite}. Hence, we obtain the following corollary.

\begin{corollary}\label{combination}
The axioms of domain semirings provide a finite equational axiomatisation of the equational theory of $\operatorname{Rel}(\compo, \bjoin, 0, 1, \D)$. 
\end{corollary}

In other words, the axioms of domain semirings are (sound and) \emph{equationally complete} for algebras of binary relations. In \cite{10.1007/978-3-540-78913-0_18}, Jipsen and Struth study the singly-generated free domain semiring. \Cref{finite} now subsumes the description of that paper, though that is not to say that these free algebras are uncomplicated objects.

The other main result of \cite{mbacke2018completeness} (Theorem 5.3.12 there) is an axiomatisation of the equational validities of our algebras of regular sets of trees. The axiomatisation used consists of the \emph{second-order} theory of star-continuous Kleene algebras, augmented with one \emph{additional} second-order axiom:
\[a \compo (\join_{b \in B}b) \compo c = \join_{b \in B}( a \compo b \compo c) \ \ \limplies\ \  a \compo (\join_{b \in B}\D (b)) \compo c = \join_{b \in B} (a \compo \D(b) \compo c),\]
where $\join$ indicates supremum. Unfortunately, this axiom is not sound for algebras of binary relations. (And so, in particular, it is not a consequence of the axioms of star-continuous Kleene algebras, which \emph{are} sound for relations.) Hence, we do not obtain an analogue of \Cref{combination} for the signature $\{\compo, \bjoin, \kstar, 0, 1, \D\}$.

\section{Automata, and closure under intersection}\label{automata}

It is well known that the set of regular languages over a finite alphabet $\Sigma$ is closed under complement with respect to $\Sigma^*$. It is clear that, over a (nonempty) finite alphabet, the regular sets of trees are not closed under the complement operation (with respect to the set of reduced trees). And, more meaningfully, this is true even in the view that a regular set $L$ represents the downward-closed set $\down L$ (since complement does not preserve downward closure). However, as we will show, the regular sets of trees \emph{are} closed under the following `intersection' operation.
\[L_1 \bmeet L_2 \coloneqq \operatorname{maximal}( \down L_1 \cap \down L_2)\]

In \cite{hellings2015relative}, and its extended journal version \cite{HELLINGS2020101467}, \emph{condition automata} are defined. They are an extension of finite-state automata designed specifically for working with relational queries that may contain $\D$ (among other tests such as range and antidomain). In this section, we will use a slightly simplified definition of condition automata.

\begin{definition}
A \defn{domain condition automaton} is a 6-tuple $(\Sigma, S, I, T, \delta, c)$, where $(\Sigma, S,\allowbreak I,\allowbreak T,\allowbreak \delta)$ is a finite-state automaton (nondeterministic, with $\epsilon$-transi\-tions permitted), and $c$ is a function that assigns a $\{\compo, \bjoin, \kstar, 0, 1, \D\}$-term (over $\Sigma$) to each state $S$.
\end{definition}

A domain condition automaton accepts a path $x_0 \overset{a_1}\to x_1 \overset{a_2}\to \dots \overset{a_n}\to x_n$ in a relational structure precisely when it is accepted by the finite-state automaton with a trace such that at each step the condition $\D(c(s))$, where $s$ is the current state, is satisfied at the corresponding vertex $x$ in the relational structure (or, to be correct, is satisfied by the pair $(x, x)$).

\begin{lemma}\label{small}
Let $t$ be a $\{\compo, \bjoin, \kstar, 0, 1, \D\}$-term, and let $T$ be a pointed labelled rooted tree, with root $r$ and point $p$. Then the pair $(r, p)$ in the relational structure $T$ satisfies $t$ if and only if there is a $T' \in \interp t$ with $T \leq T'$.
\end{lemma}

\begin{proof}
This is just a specialisation of \Cref{path} (and the equivalence of the condition $T \leq T'$ with the existence of a homomorphism from $T'$ into $T$).
\end{proof}

\begin{proposition}
Let $\Sigma$ be an alphabet and let $L_1$ and $L_2$ be two sets of reduced pointed  $\Sigma$-labelled finite rooted trees. If $L_1$ and $L_2$ are regular, then $L_1 \bmeet L_2$ is regular.
\end{proposition}

\begin{proof}
Let $L_1 = \interp {t_1}$ and $L_2 = \interp {t_2}$. Then by \cite[Proposition 5]{hellings2015relative}, there is a domain condition automaton $\mathcal A_1$ such that for any pointed labelled rooted tree $T$, with root $r$ and point $p$, the pair $(r, p)$ satisfies $t_1$ if and only if the path from $r$ to $p$ is accepted by $\mathcal A_1$. Similarly, there is such a domain condition automaton $\mathcal A_2$ for $t_2$. By \cite[Proposition 6]{hellings2015relative}, there is a domain condition automaton $\mathcal A$ that accepts a path if and only if that path is accepted by both $\mathcal A_1$ and $\mathcal A_2$. Then using \cite[Proposition 5]{hellings2015relative} in the other direction, applied to $\mathcal A$, there is a $\{\compo, \bjoin, \kstar, 0, 1, \D\}$-term $t$ such that for any pointed labelled rooted tree $T$, with root $r$ and point $p$, the pair $(r, p)$ satisfies $t$ if and only if it satisfies both $t_1$ and $t_2$.

Applying \Cref{small}, we have (for \emph{reduced} $T$) that $T \in \down \interp t$ if and only if $(r, p)$ satisfies $t$ if and only if  $(r, p)$ satisfies both $t_1$ and $t_2$ if and only if $T \in \down \interp {t_1}$ and $T \in \down \interp {t_2}$ if and only if  $T \in \down L_1 \cap \down L_2$. Hence $\down \interp t = \down L_1 \cap \down L_2$. So \[\interp t = \operatorname{maximal}(\down \interp t) = \operatorname{maximal}(\down L_1 \cap \down L_2) = L_1 \bmeet L_2.\qedhere\]
\end{proof}

Let $(\mathcal O(\mathcal R_\Sigma), {\cup}, {\cap})$ be the lattice of downward-closed subsets of $\mathcal R_\Sigma$. We now know that the sets of the form $\down L$, for regular $L$, form a sublattice of $(\mathcal O(\mathcal R_\Sigma), {\cup}, {\cap})$. Given that $(\mathcal O(\mathcal R_\Sigma), {\cup}, {\cap})$ is a Heyting algebra, this raises the further question of whether the sets of the form $\down L$ are closed under the implication operation of $(\mathcal O(\mathcal R_\Sigma), {\cup}, {\cap})$.

\begin{problem}\label{Heyting}
Are the regular sets of reduced trees closed under the following implication operation?\footnote{Thanks go to one of the anonymous referees for posing this question.}
\begin{align*}L_1 \limplies L_2 &\coloneqq \operatorname{maximal} \{ \bigcup\{K \in \mathcal O(\Sigma) \mid  K \cap \down L_1 \subseteq \down L_2\}\}
\end{align*}
\end{problem}

We conjecture that the answer to this problem is yes.

We pose one additional problem, relating to extra-order-theoretic \emph{composition} structure of the regular sets.

\begin{problem}
Are the regular sets of reduced trees closed under the following residuation operations?
\begin{align*}L_1 \mathbin{\setminus} L_2 &\coloneqq \operatorname{maximal} \{ T \in \mathcal R_\Sigma \mid \forall S \in \down L_1,\ S \compo T \in \down L_2\}\\
L_1 \mathbin{\minusset} L_2 &\coloneqq \operatorname{maximal} \{ T \in \mathcal R_\Sigma \mid \forall S \in \down L_2,\ T \compo S \in \down L_1\}
\end{align*}
\end{problem}

\section{Decidability of equational theory}\label{section:decide}

In this section we describe how to decide the validity of a $\{\compo, \bjoin, \allowbreak\kstar,\allowbreak 0, 1, \D\}$-equation with respect to relational semantics.

First, we give a definition of condition automata closer to that found in \cite{hellings2015relative} and \cite{HELLINGS2020101467}. Recall (\Cref{tests}) that $\A$ is a unary function symbol whose relational interpretation is the antidomain operation.

\begin{definition}
A \defn{condition automaton} is a 6-tuple $(\Sigma, S, I, T, \delta, c)$, where $(\Sigma, S,\allowbreak I,\allowbreak T,\allowbreak \delta)$ is a finite-state automaton (nondeterministic, with $\epsilon$-transitions permitted), and $c$ is a function that assigns a $\{\compo, \bjoin, \kstar, 0, 1, \A\}$-term (over $\Sigma$) to each state $S$.
\end{definition}

Acceptance for condition automata is defined just like acceptance for domain condition automata, where the symbol $\D$ is now shorthand for two applications of $\A$. 

\begin{theorem}\label{decidability}
The equational theory of the class of algebras of binary relations of the signature $\{\compo, \bjoin, \kstar, 0, 1, \D\}$ is decidable.
\end{theorem}

\begin{proof}
Let $s$ and $t$ be $\{\compo, \bjoin, \kstar, 0, 1, \D\}$-terms, and let $\Sigma$ be the set of variables appearing in either $s$ or $t$. Now $s = t$ is valid in $\operatorname{Rel}(\compo, \bjoin, \kstar, 0, 1, \D)$ if and only if the pointed $\Sigma$-labelled rooted trees satisfying $s$ are precisely those satisfying $t$---we proved the stronger statement involving \emph{reduced} trees. In the equivalence just stated, we can temporarily use the usual graph-theoretic definition of a pointed $\Sigma$-labelled rooted tree; in particular we do not limit to finite trees. According to \cite[Proposition 5]{hellings2015relative}, there are (domain) condition automata $\mathcal A_s$ and $\mathcal A_t$ that accept precisely the \emph{finite} pointed trees satisfying $s$ and $t$ respectively. But in fact the finiteness condition plays no role, and hence can be dropped. The same remark can be made for the following statements and we will make no further mention of it. By \cite[Corollary 3]{hellings2015relative}, there are condition automata $\mathcal A_{s - t}$ and $\mathcal A_{t-s}$ such that a pointed tree is accepted by $\mathcal A_{s - t}$ precisely if it is accepted by $\mathcal A_s$ but not by $\mathcal A_t$, and a tree is accepted by $\mathcal A_{t - s}$ precisely if it is accepted by $\mathcal A_t$ but not $\mathcal A_s$. By \cite[Proposition 5]{hellings2015relative}, there exist $\{\compo, \bjoin, \kstar, 0, 1, \A\}$-terms $\tau_{s-t}$ and $\tau_{t-s}$ such that the pointed trees satisfying $\tau_{s-t}$ and $\tau_{t-s}$ are precisely those accepted by $\mathcal A_{s - t}$ and $\mathcal A_{t-s}$ respectively. Hence $s = t$ is valid in $\operatorname{Rel}(\compo, \bjoin, \kstar, 0, 1, \D)$ if and only if the sets of pointed $\Sigma$-labelled rooted trees satisfying $\tau_{s-t}$ and $\tau_{t-s}$ are both empty. Finally, we reduce the problem of deciding if a $\{\compo, \bjoin, \kstar, 0, 1, \A\}$-term $\tau$ is satisfiable by a pointed labelled rooted tree to the problem of deciding the satisfiability of a formula of propositional dynamic logic.\footnote{See, for example, \cite{FISCHER1979194}, for a description of propositional dynamic logic, including the notion of a \emph{regular frame}.} This latter problem is known to be decidable (in {\sf EXPTIME} \cite{PRATT1980231}), so then we are done. See \Cref{flow} for a summary of the transformations we have outlined.
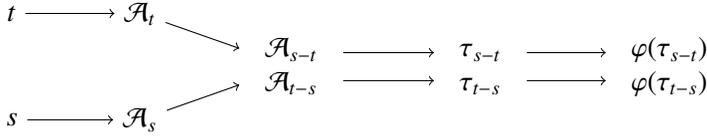
\begin{figure}[H]\centering
\begin{tikzpicture}[rotate=90]
\node (A) at (0.3, -.3) {$s$};
\node (B) at (1.7, -0.3) {$t$};
\node (C) at (0.3, -2) {$\mathcal A_s$};
\node (D) at (1.7, -2) {$\mathcal A_t$};
\node (E) at (1, -4) {\begin{tabular}{c}$\mathcal A_{s - t}$\\$\mathcal A_{t - s}$\end{tabular}};
\node (F) at (1, -6.5) {\begin{tabular}{c}$\tau_{s - t}$\\$\tau_{t - s}$\end{tabular}};
\node (G) at (1, -9) {\begin{tabular}{c}$\varphi(\tau_{s - t})$\\$\varphi(\tau_{t - s})$\end{tabular}};

\draw[->] (A) -- (C);
\draw[->] (B) -- (D);
\draw[->] (C) -- (E);
\draw[->] (D) -- (E);
\draw[->] ($(E.east) + (.5em, 0)$) -- ($(F.west) + (.5em, 0)$);
\draw[->] ($(E.east) - (.5em, 0)$) -- ($(F.west) - (.5em, 0)$);
\draw[->] ($(F.east) + (.5em, 0)$) -- ($(G.west) + (.5em, 0)$);
\draw[->] ($(F.east) - (.5em, 0)$) -- ($(G.west) - (.5em, 0)$);
\end{tikzpicture}
\caption{Transformations used to decide equality of $t$ and $s$}\label{flow}
\end{figure}
We define a translation from $\{\compo, \bjoin, \kstar, 0, 1, \A\}$-terms to (propositional varia\-ble-free) formulas of propositional dynamic logic as follows. First we define the translation $P$ from $\{\compo, \bjoin, \kstar, 0, 1, \A\}$-terms to program terms by structural induction as follows.
\begin{align*}
P(a) &\coloneqq a\\
P(s \compo t) &\coloneqq P(s) \compo P(t)\\
P(s \bjoin t) &\coloneqq P(s) \bjoin P(t)\\
P(t^*) &\coloneqq P(t)^*\\
P(0) &\coloneqq \bot?\\
P(1) &\coloneqq \top?\\
P(\A(t)) &\coloneqq (\neg P(t))?
\end{align*}
Then we simply define the translation $\varphi(t)$ of a $\{\compo, \bjoin, \kstar, 0, 1, \A\}$-term $t$ to be $\<P(t)\>\top$. It is clear that for any given regular frame, satisfiability of $\varphi(t)$ is equivalent to satisfiability of $t$ on the corresponding relational structure. Since propositional dynamic logic has the \emph{tree-model property} (any frame with selected point can be `unwound' to an equivalent labelled rooted tree), we obtain the equivalence of satisfiability of $\varphi(t)$ on regular frames and satisfiability of $t$ on tree-based relational structures. But if $t$ is satisfied by a pair $(x, y)$ of vertices in a tree-based relational structure, then clearly $y$ is a descendent of $x$, and $t$ is satisfied by the tree rooted at $x$ and having point $y$. (And conversely, satisfaction by a pointed rooted tree implies satisfaction by a tree-based relational structure.) Hence we have the required equivalence between satisfaction of $t$ with respect to pointed labelled rooted trees and satisfaction of $\varphi(t)$ with respect to regular frames.
\end{proof}

The procedure described in the proof of \Cref{decidability} hardly seems efficient. The best upper bound that can be obtained from it is a {$\sf 3EXPTIME$} bound. The constructions of $\mathcal A_{s - t}$ and $\mathcal A_{t-s}$ rely on determinisations involving a subset construction, so can result in an exponential increase in problem size. Likewise, construction of the terms $\tau_{s-t}$ and $\tau_{t-s}$ from automata can also add an exponent in general. Lastly, we already mentioned that the final step, deciding satisfiability of propositional dynamic logic formulas, is in {\sf EXPTIME}, and in fact this problem also has an exponential time lower bound (no $\mathcal O(2^{n^\varepsilon})$ algorithm for any $\varepsilon < 1$ \cite{FISCHER1979194}).

We finish with the obvious problem.

\begin{problem}
Determine the precise complexity of deciding validity of $\{\compo, \bjoin, \allowbreak\kstar,\allowbreak 0, 1, \D\}$-equations with respect to relational semantics.
\end{problem}

\appendix
\section{Appendix}

\begin{theorem}
The class $\operatorname{Rel}(\compo, \bjoin, \kstar, 0, 1)$ is not closed under elementary equivalence.
\end{theorem}

\begin{proof}
By the definition of $\kstar$ on relations, algebras in $\operatorname{Rel}(\compo, \bjoin, \kstar, 0, 1)$ validate $a^* = \operatorname{sup}_{i\in \mathbb N} a^i$ (with respect to the ordering $a \leq b \iff a + b = b$). Take any relation $R$ whose star $R^*$ differs from all finite approximants $\operatorname{sup}_{i <n} R^i$---for example take $R$ to be the immediate-successor relation in $\mathbb N$. Let $\algebra A$ be any algebra of relations containing $R$---for example the algebra generated by $R$. Let $U$ be any non-principal ultrafilter on $\mathbb N$, and let $\bar R$ be the element of the ultrapower ${\algebra A^\mathbb N} / {U}$ represented by the constant sequence $(R, R, \dots)$. Then $(\bar R)^*$ is represented by $(R^*, R^*, \dots)$, but this is not a supremum for $\{R^i \mid i \in \mathbb N\}$ since the strictly smaller element represented by $(1, {1+ R}, 1 + R + R^2, \dots)$ is also an upper bound. Hence the ultrapower does not validate $a^* = \operatorname{sup}_{i\in \mathbb N} a^i$, so $\operatorname{Rel}(\compo, \bjoin, \kstar, 0, 1)$ is not closed under ultrapowers. It follows by \L o\'s's theorem \cite{LOS195598} that $\operatorname{Rel}(\compo, \bjoin, \kstar, 0, 1)$ cannot be closed under elementary equivalence.
\end{proof}

\bibliographystyle{amsplain}
\bibliography{../brettbib}

\end{document}